\documentclass{article}

\usepackage[utf8]{inputenc}
\usepackage{microtype}
\usepackage{xcolor}
\usepackage{color}
\usepackage{amssymb, latexsym,pdfsync,amsmath,amsthm, ulem,hyperref,graphicx}

\newtheorem{theorem}{Theorem}
\newtheorem{definition}{Definition}
\newtheorem{lemma}{Lemma}
\newtheorem{proposition}{Proposition}
\newtheorem{question}{Question}
\newtheorem{remark}{Remark}
\newtheorem{corollary}{Corollary}
\newtheorem{example}{Example}

\newcommand{\NN}{\mathbb{N}}

\newcommand{\F}{\mathbb{F}}

\newcommand{\fq}{\mathbb{F}_q}
\newcommand{\fqn}{\mathbb{F}_{q^n}}
\newcommand{\cB}{\mathcal{B}}
\newcommand{\ann}{\mathrm{ann}}
\newcommand{\Tr}{\mathrm{Tr}}
\newcommand{\ord}{\mathrm{ord}}
\newcommand{\rmv}[1]{}

\title{A note on depth-$b$ normal elements}
\author{John Sheekey \and David Thomson}

\begin{document}
\maketitle
\begin{abstract} 
In this paper we study elements $\beta \in \fqn$ having normal $\alpha$-depth $b$; that is, elements for which $\beta, \beta - \alpha, \ldots, \beta-(b-1)\alpha$ are simultaneously normal elements of $\fqn$ over $\fq$. In~\cite{GoveGary}, the authors present the definition of normal $1$-depth but mistakenly present results for normal $\alpha$-depth for some fixed normal element $\alpha \in \fqn$. We explain this discrepancy and generalize the given definition of normal $(1-)$depth from~\cite{GoveGary} as well as answer some open questions presented in~\cite{GoveGary}.
\end{abstract}



\maketitle
\section{Introduction and notation} 

Throughout this document, we use the following standard notation. Let $p$ be a prime and let $q$ be a power of $p$, the finite field of $q$ elements is denoted $\fq$, and the finite degree $n$ extension of $\fq$ is denoted $\fqn$. The (relative) trace function is denoted $\Tr_{\fqn:\fq}\colon \fqn\to\fq$. We remark that the trace function is onto, and for any $k\not\equiv0\pmod{p}$, the element $k\alpha$ is also normal.  For any positive integer $n$, denote by $e = v_p(n)$, the \emph{$p$-ary valuation} of $n$; that is the largest integer $e$ such that $p^e$ divides $n$ but $p^{e+1}$ does not divide $n$. We also denote by $\tau = p^e$; specifically, $\tau = 1$ ($e=0$) if $\gcd(p,n) = 1$.

In Section~\ref{sec:Frob}, we derive conditions for elements to be normal that we will use later in the paper. In Section~\ref{sec:depthb}, we correct and generalize the notion of normal elements of depth $b$ from~\cite{GoveGary}. Also motivated by~\cite{GoveGary}, in Section~\ref{sec:conjugates} we observe that depth is not necessarily invariant under conjugation, and further analyze the depth of the conjugates of normal elements. 

\section{Finite fields as Frobenius modules}\label{sec:Frob}
In this section, we follow~\cite{LS, prim-1-normal} and introduce finite fields as Frobenius modules. Let $\sigma_q\colon \overline{\fq}\to\overline{\fq}$ denote the Frobenius $q$-automorphism. Clearly, $\sigma_q$ fixes $\fq$ and for any $n > 0$ and $\alpha \in \overline{\fq}$, $\sigma_q^n(\alpha) = \alpha$ if and only $\alpha \in \fqn$. Moreover, the Galois group of $\fqn$ over $\fq$ is cyclic of order $n$ and generated by $\sigma_q$. 

Let $\alpha \in \fqn$ and let $\cB$ consist of the Galois orbit of $\alpha$; that is, $\cB = \{\alpha, \alpha^q, \ldots, \alpha^{q^{n-1}}\}$. If $\cB$ is a linearly independent set, then $\alpha$ is a \emph{normal element} of $\fqn$ and $\cB$ is a \emph{normal basis} of $\fqn$ over $\fq$. We also call $\alpha$ a \emph{cyclic vector} for $\fqn$ as a vector space over $\fq$. 

For $f(x) = \sum_{i=0}^{m} a_i x^i$, denote the action of $f$ on $\overline{\fq}$ by 
\[ f \circ \alpha = f(\sigma_q)(\alpha) = \sum_{i=0}^{m} a_i \alpha^{q^i}. \]
Clearly, $(f + g)\circ \alpha = f\circ\alpha + g\circ\alpha$ for any $f,g \in \fq[x]$, and $(x^n-1) \circ \alpha = 0$ if and only if $\alpha \in \fqn$. Moreover, $(fg)\circ \alpha = f\circ(g\circ \alpha)$, so that if $f \circ \alpha = 0$ for any $\alpha \in \fqn$, then $f$ divides $x^n-1$.

\begin{definition} \ \\
\begin{enumerate}
\item For any $\alpha \in \fqn$, define the \emph{annihilator} of $\alpha$ as the polynomial $\ann_\alpha \in \fq[x]$ of smallest degree such that $\ann_\alpha \circ \alpha = 0$. 
\item For any $f \in \fq[x]$, define $\ker(f) = \{\alpha \in \fqn\colon \ann_\alpha = f\}$, the set of elements of $\fqn$ annihilated by $f$ under $\circ$. 
\end{enumerate}
\end{definition}

Observe that $\ann_\alpha$ annihilates any linear combination of Galois conjugates of $\alpha$. We have $\ker(x^n-1) = \fqn$ and $\ann_\alpha(x)$ divides $x^n-1$ for any $\alpha$. Moreover, $\alpha$ is a normal element of $\fqn$ over $\fq$ if and only if $\ann_\alpha(x) = x^n-1$ by linear independence of the conjugates of $\alpha$. We summarize these observations in Proposition~\ref{prop:numnormal}.

\begin{proposition}\label{prop:numnormal}
For any prime power $q$, the number of normal elements of $\fqn$ over $\fq$ is given by $\Phi_q(x^n-1)$, where $\Phi_q$ is Euler's totient function over $\fq$; that is, $\Phi_q(x^n-1)$ is the number of polynomials in $\fq[x]$ of degree less than $n$ that are relatively prime with $x^n-1$. 
\end{proposition}

Existence of normal elements can be gleaned directly from Proposition~\ref{prop:numnormal}, since $\Phi_q(x^n-1)$ is nonzero for all $n \geq 1$. 

We now introduce a map central to the remainder of this work. Suppose $\alpha \in \fqn$ is normal and define the map $\phi_\alpha\colon \fq[x]\to \fqn$ by $\phi(f) = f\circ\alpha$. Then $\ker(\phi_\alpha) = ( x^n-1 )$, since $\alpha$ is normal; similarly $\phi_\alpha$ is onto since the set $\cB_{\alpha} = \{\alpha, \alpha^q, \ldots, \alpha^{q^{n-1}}\}$ is a basis. Hence $\fqn\cong \fq[x]/(x^n-1)$ as Frobenius modules. We will abuse notation and refer to this isomorphism also as $\phi_\alpha$.

Let $g(x) = \sum_{i=0}^{n-1} g_i x^{i} \in \fq[x]$, and $\beta = \phi_\alpha(g) =\sum_{i=0}^{n-1} g_i \alpha^{q^i}$. Then $\beta^q = \sum_{i=0}^{n-1} g_{i-1} \alpha^{q^i}$. Thus $\phi_\alpha^{-1}(\beta^q) = x\phi_\alpha^{-1}(\beta) \mod (x^n-1)$. Thus the Frobenius action on $\fq[x]/(x^n-1)$ is induced by $\overline{\sigma_q}(g):= xg(x)$, with $\overline{\sigma_q} = \phi_\alpha \sigma_q \phi_\alpha^{-1}$.

We exploit the decomposition of $\fq[x]/(x^n-1)$ as a Frobenius module. We follow the treatment in~\cite{Steel}. Let $e = \nu_p(n)$ be the valuation of $n$ at $p$ and let $x^n-1 = f_1^{e_1} \cdots f_r^{e_r}$ be the primary factorization of $x^n-1$, then $e_i = p^{e} = \tau$ for all $i = 1, \ldots, r$. In particular, $\tau = 1$ if $\gcd(p,n) = 1$. Denote by $\overline{V_i} = \fq[x]/(f_i^{\tau})$, then

\begin{equation}
 \fq[x]/(x^n-1) \cong \bigoplus_{i=1}^r \overline{V_i}.\label{eqn:primary_decomposition}
\end{equation}

Explicitly, we write the image of $g$ in $\bigoplus_{i=1}^r \overline{V_i}$ as $(g \mod f_1^{\tau},\ldots,g \mod f_r^{\tau})$. We abuse notation slightly and write $V_i=\phi_\alpha(\overline{V_i})$.
 
 \begin{equation}
\bigoplus_{i=1}^r V_i\cong \fqn \cong \fq[x]/(x^n-1) \cong \bigoplus_{i=1}^r \overline{V_i}.\label{eqn:primary_decomposition2}
\end{equation}

Equation~\eqref{eqn:primary_decomposition2} is the \emph{primary decomposition} of $\fqn$ as a Frobenius module. Moreover, we observe that each $V_i$ is stable under $\sigma_q$. 

\begin{proposition}\label{prop:normalelt}
Let $\alpha$ be a normal element of $\fqn$, and suppose $\beta=\phi_\alpha(g(x))$. Then $\ann_\beta = \frac{x^n-1}{\gcd(x^n-1,g(x))}$, and $\beta$ is normal if and only if $\gcd(x^n-1,g(x))=1$. Furthermore, $V_i=\ker(f_i^{\tau})$.
\end{proposition}

\begin{proof}
Let $f(x) = \sum_{i=0}^{m} a_i x^i$. Then $f\circ \beta = 0$ if and only if $f(x)g(x)\in (x^n-1)$. The smallest degree polynomial satisfying $f(x)g(x)\in (x^n-1)$ is clearly $\frac{x^n-1}{\gcd(x^n-1,g(x))}$, as claimed. Since $\beta$ is normal if and only if $\ann_\beta(x) = x^n-1$, $\beta$ is normal if and only if $\gcd(x^n-1,g(x))=1$.

Now $\beta \in \phi_\alpha(\overline{V_i})$ if and only if $f_j^{\tau}$ divides $g(x)$ for all $j\ne i$, which occurs if and only if $f_i(x)^{\tau}g(x)\in (x^n-1)$, if and only if $f_i^{\tau}\circ \beta =0$, if and only if $\beta \in \ker(f_i^{\tau})$. 
\end{proof}

We summarise the characterisations of normal elements here.

\begin{proposition}\label{prop:normalequiv}
Let $\alpha$ be a normal element of $\fqn$, and suppose $\beta=\phi_\alpha(g(x))$. Let $x^n-1 = f_1^{\tau} \cdots f_r^{\tau}$, with the $f_i$ being distinct irreducible polynomials in $\fq[x]$. Let $g_i=g\mod f_i^{\tau}$, and $\beta=\sum_{i=1}^r \beta_i$ for $\beta_i\in V_i$. Then the following are equivalent:

\begin{enumerate}
\item
$\beta$ is normal,
\item
$\gcd(x^n-1,g(x))=1$,
\item
$\gcd(f_i,g_i)=1$ for each $i$.
\item 
$\ann_{\beta_i}=f_i^{\tau}$ for each $i$,
\item
$\beta_i\in \ker(f_i^{\tau})\backslash  \ker(f_i^{\tau-1})$ for each $i$,
\end{enumerate}
\end{proposition}

\begin{proof}
($1. \iff 2.$) This is Proposition~\ref{prop:normalelt}. 

($2. \iff 3.$) Let $g_i = g\mod f_i^\tau$, then $g = h f_i^\tau + g_i$ for some $h \in \fq[x]$. Then (the irreducible) $f_i$ divides $g_i$ for some $1 \leq i \leq r$ if and only if $f_i$ divides $g$, contradicting $\gcd(x^n-1, g(x)) = 1$. 

($3. \iff 4.$) Let $g_i = g \mod f_i^\tau$ with $\beta_i = \phi_\alpha(g_i) \in V_i$. Clearly $\beta_i \in V_i$ if and only if $\beta_i \in \ker(f_i^\tau)$, so $\ann_{\beta_i} = f_i^k$ for $1 \leq k \leq \tau$. Now, $f_i^k\circ \beta_i = 0$ if and only if $f_i^k g_i \circ \alpha = 0$ if and only if $f_i^k g_i \in (x^n-1)$. Now $\gcd(f_i, g_i) = 1$ if and only if $k = \tau$ for all $i$. 

($4. \iff 5.$) By the minimality of $\ann_{\beta_i}$, we have $\ann_{\beta_i} = f_i^\tau$ if and only if $\beta_i \in \ker(f_i^\tau)$ and $\beta_i \notin \ker(f_i^{\tau-1})$. \qedhere
\end{proof}
	
If $\gcd(p,n) =1$, then $\tau=1$, and thus we get the following.

\begin{corollary}\label{cor:normalPD}
Let $\gcd(p,n) =1$ and let $\beta = \beta_1 + \beta_2 + \cdots + \beta_r$ with $\beta_i \in V_i$, then $\beta$ is a normal element if and only if $\prod_{i=1}^r \beta_i \neq 0$. 
\end{corollary}

\section{Depth-$b$ normal elements}\label{sec:depthb}

\begin{definition}\label{def:depth}
Let $b\in \NN$ with $b\leq p$. If $\beta \in \fqn$ is such that $\beta, \beta - \alpha, \ldots, \beta-(b-1)\alpha$ are normal elements of $\fqn$ over $\fq$ for some $\alpha\in \fqn$, then we say that $\beta$ has \emph{normal $\alpha$-depth $b$}.  
\end{definition}

In \cite{GoveGary}, the authors introduced normal depth, where the definition was for $\theta=1$. However the results in \cite{GoveGary} are in fact referring to normal $\alpha$-depth, for some fixed normal element $\alpha$. We will explain the discrepancy below, and consider the more general problem.

We remark that Defintion~\ref{def:depth} can be extended for $b \geq p$ when $q$ is a power of $p$ by imposing an ordering on the elements of $\fq$ (or even further still, on $\fqn$). Since~\cite{GoveGary} and Section~\ref{sec:conjugates} are mostly concerned with depth $2$, we will not treat these sorts of extensions in this work.

We recap (and generalize) the main question from~\cite{GoveGary}.

\begin{question}\label{question}
To what extent do the conjugates of an element $\beta$ having normal $\alpha$-depth $b$ also have normal $\alpha$-depth $b$? 
\end{question}

In particular in \cite{GoveGary}, they focus on normal depth $2$ and search for \emph{lonely elements}: that is, normal elements of depth $2$ having a conjugate that fails to have normal depth $2$.

\begin{lemma}
Without loss of generality, fix a normal element $\alpha$ of $\fqn$ satisfying $\Tr_{\fqn:\fq}(\alpha)=n/\tau$, since if $\alpha'$ is any normal element with $\Tr_{\fqn:\fq}(\alpha') = k \neq 0$, the element $\alpha = \alpha' \frac{\tau}{nk}$ is normal (since $\tau/n \not\equiv 0\pmod{p}$). Then,
\begin{enumerate}
\item 
$\phi_\alpha^{-1}(\alpha^{q^i})=x^i$ and $\phi_\alpha^{-1}(1) = (\tau/n)\frac{x^n-1}{x-1}$, 
\item 
the image of $\alpha$ in $\bigoplus_{i=1}^r \overline{V_i}$ is $(1,1,\ldots,1)$, and the image of $1$ is $((x-1)^{\tau-1},0,\ldots,0)$.
\end{enumerate}
\end{lemma}
\begin{proof}
\begin{enumerate}
\item For $f(x) = \sum_{i=0}^{m} a_i x^i$, we have $\phi_\alpha(f) = f \circ \alpha = \sum_{i=0}^{m} f_i \alpha^{q^i}$. Hence, $x^i \circ \alpha = \alpha^{q^i}$ or $\phi_\alpha^{-1}(\alpha^{q^i}) = x^i$.  Similarly, 
\[ 
\Tr_{\fqn:\fq}(\alpha) = \left(\sum_{i=0}^{n-1} x^i\right)\circ \alpha = n/\tau,
\]
so by linearity, $\phi_\alpha^{-1}(1) = (\tau/n) \sum_{i=0}^{n-1} x^i = (\tau/n) \frac{x^n-1}{x-1}$. 
\item Since $\alpha = \phi_\alpha(1)$, we have $g_i = g\mod f_i^\tau = 1$ for all $1 \leq i \leq r$. Similarly, $1 = (\tau/n)\phi_\alpha (\frac{x^n-1}{x-1})$, and with $\tau = p^{\nu_p(n)}$ and by linearity of Frobenius,
\begin{align*}
\frac{x^n-1}{x-1} &= \sum_{i=0}^{n-1} x^i \equiv (n/\tau) \sum_{i=0}^{\tau-1} x^i \mod (x^\tau -1) \\ 
                  &= (n/\tau) \frac{x^\tau-1}{x-1} = (n/\tau)(x-1)^{\tau-1}. \qedhere
\end{align*}
\end{enumerate}
\end{proof}

\begin{proposition}
Let $\alpha \in \fqn$ be normal. An element $\beta=\phi_\alpha(g(x))$ has normal $\alpha$-depth $b$ if and only if $\gcd(x^n-1,g(x)-c)=1$ for all $c\in \{0,\ldots,b-1\}$.
\end{proposition}
\begin{proof}
The proof is immediate from the linearity of $\phi$ and from Proposition~\ref{prop:normalequiv}, Remark 2. 
\end{proof}
In \cite{GoveGary}, the number $\#\{g:\gcd(x^n-1,g(x)-c)=1~\forall~c\in \{0,\ldots,b-1\}\}$ was defined as $\Phi_b(x^n-1)$.
 
\begin{theorem}
Let $\alpha \in \fqn$ be normal with $\Tr_{\fqn:\fq}(\alpha) = \tau/n$, let $e = \nu_p(n)$, and let $\beta=\phi_\alpha(g(x))$ also be normal. Then 
\begin{enumerate}
\item if $e > 0$, then $\beta$ has normal $1$-depth $p$; moreover, $\beta - c$ is normal for all $c \in \fq$, 
\item if $e = 0$, then $\beta$ has normal $1$-depth $b$ if and only if $g(1) \geq b$ (under a suitable implicit ordering of the elements of $\fq$). In particular, $\beta - c$ is normal if and only if $g(1) \neq c$. 
\end{enumerate}
\end{theorem}

\begin{proof}
Let $g_i=g\mod f_i^{\tau}$. Then the image of $\beta-c$ in $\bigoplus_{i=1}^r \overline{V_i}$ is $(g_1 -c (x-1)^{\tau-1},g_2,\ldots,g_r)$. 

If $e>0$ and $\beta$ is normal, then $\gcd(g_1,(x-1)^\tau)=1$. If $\beta-c$ is not normal, then $(x-1)$ divides $g_1 -c (x-1)^{\tau-1}$, implying $(x-1)$ divides $g_1$, a contradiction. Thus $\beta-c$ is normal for all $c\in \fq$.

If $e=0$, then $g_1=g(1)$, and the image of $g$ is $(g(1) -c ,g_2,\ldots,g_r)$. By Corollary~\ref{cor:normalPD}, $\beta$ is normal if and only if $g_i\ne0$ for each $i$. Hence, $\beta-c$ is {\it not} normal if and only if $g(1)= c$. 
\end{proof}

In \cite{GoveGary} the authors mistakenly state that the number of elements having normal $1$-depth $b$ is equal to $\Phi_b(x^n-1)$. This assumably arose by the erroneous assumption that $\phi_\alpha(1)=1$. Instead, since $\phi_\alpha(1) = \alpha$, $\Phi_b(x^n-1)$ refers to the number of elements having normal $\alpha$-depth $b$, and so for the remainder of this paper we focus on this case as well.

\section{Conjugates: Lonely and Sociable elements}\label{sec:conjugates}

Throughout this section, we use the notation from Section~\ref{sec:depthb}; in particular, $x^n-1 = (f_1 \cdots f_r)^{\tau}$ where $n = \tau m$ with $\gcd(m,\tau) = 1$, and $f_i$ is irreducible for $1 \leq i \leq r$. Suppose $\beta=\phi_\alpha(g(x))$ has normal $\alpha$-depth $b$. We consider the normal $\alpha$-depth of its conjugates. Recall that $\beta^{q^i}=\phi_\alpha(x^i g(x))$. Thus we need to consider the common divisors of $x^ig(x)-c$ with $x^n-1$, or equivalently $g(x)-cx^i$ with $x^n-1$.

\begin{definition}\label{def:lonely-social}
An element $\beta\in \fqn$ is \emph{$(\alpha,b)$-lonely} if $\beta$ has normal $\alpha$-depth $b$, but $\beta^{q^i}$ does not have normal $\alpha$-depth $b$ for some $i$. If $\beta^{q^i}$ has normal $\alpha$-depth $b$ for all $i$, we say that $\beta$ is \emph{$(\alpha,b)$-sociable}.
\end{definition}

Similar to Proposition~\ref{prop:normalequiv}, we have a number of equivalent characterizations of sociable elements. 

\begin{theorem}\label{thm:tfae-sociable}
Let $x^n-1 = f_1^\tau f_2^\tau \cdots f_r^\tau$ with $f_i$ irreducible, $1 \leq i \leq r$. Let $\beta\in \fqn$ with $g(x)=\phi_\alpha^{-1}(\beta)$ and let $g_i = g\mod f_i^\tau$. Then the following are equivalent:

\begin{enumerate}
\item
$\beta$ is $(\alpha,b)$-sociable,
\item
$\gcd(x^n-1,g(x)-cx^j)=1$ for all $j\in \{0,\ldots,n-1\}$, $c\in \{0,\ldots,b-1\}$,
\item
$\gcd(f_i,g_i - cx^j)=1$ for all $i\in \{1,\ldots,r\}$, $j\in \{0,\ldots,n-1\}$, $c\in \{0,\ldots,b-1\}$,
\item
$g(\theta)\notin \{c\theta^j:c \in \{0,\ldots,b-1\}, j\in \{0,\ldots,n-1\}\}$ and $\theta$ a root of $x^n-1$.
\end{enumerate}
\end{theorem}
\begin{proof}
The equivalence of items $1.$, $2.$, $3.$ come directly from applying Proposition~\ref{prop:normalequiv} to Definition~\ref{def:lonely-social}. Here we prove only $3. \iff 4.$ 

Suppose $\gcd(f_i, g_i-cx^j) \neq 1$ for some $1 \leq i \leq r$, $0 \leq j \leq n-1$, which occurs if and only if $f_i(\theta_i) = g_i(\theta_i) - c\theta_i^j = 0$ for some $\theta_i \in \F_{q^{\deg(f_i)}}$; that is, $g_i(\theta_i) = c\theta_i^j$. The fourth equivalence follows, since $g_i = g\mod f_i^\tau$, so $g(\theta_i) = g_i(\theta_i).$
\end{proof}

 The number of $\beta$ that are $(\alpha,b)$-sociable is the number of $g$ satisfying the conditions on their roots given in the fourth equivalence of Theorem~\ref{thm:tfae-sociable}. 

\begin{lemma}\label{lem:trivial}
Let $x^n-1 = f_1^\tau \cdots f_r^{\tau}$ and let $\theta_i$ be a root of $f_i$, $1 \leq i \leq r$. Then there are exactly $q^{\deg(f_i)}$ possible values for $g(\theta_i)$ for $g\in \fq[x]$. Furthermore, let $\theta_{ij} = \theta_i^{q^j}$ for $j = 0, 1, \ldots, \deg(f_i)$ be the roots of $f_i$ in $\fq(\theta_i)$, and fix $\gamma_{i} \in \fq(\theta_i)$, $1 \leq i \leq r$. Then there exist precisely $q^{\frac{n(\tau-1)}{\tau}}$ polynomials $g$ of degree at most $n$ with $g(\theta_{ij}) = \gamma_{i}^{q^j}$ for all  $1 \leq i \leq r$, $0 \leq j \leq \deg(f_i)-1$.
\end{lemma}

\begin{proof}
Clearly $g(\theta_i)\in \fq(\theta_i)=\mathbb{F}_{q^{\deg(f_i)}}$, and so there are at most $q^{\deg(f_i)}$ possible values for $g(\theta_i)$. As $g$ has coefficients in $\fq$, we have that $g(\theta_i^{q^j})=g(\theta_i)^{q^j}$ for any $j$.

With $n = n_0\tau$, two polynomials $g$ and $h$ in $\fq[x]$ agree on all $n_0$-th roots of unity if and only if $f_1f_2 \cdots f_r$ divides $g-h$. As $\deg(f_1f_2 \cdots f_r) = n_0$, there are $q^{n-n_0} = q^{n_0(\tau-1)}$ such polynomials $h$ of degree at most $n$.
\end{proof}

For $\beta$ that are $(\alpha, b)$-sociable, Theorem~\ref{thm:tfae-sociable} provides a number of forbidden values for $g(\theta_i)$. The precise number of forbidden values that ensure that $\beta$ is $(\alpha,b)$-sociable is complicated in general, but we can solve it completely in some cases left open in \cite{GoveGary}.

\begin{proposition}
The number of elements in $\fqn$ that are $(\alpha,b)$-sociable is {\it at most}
\[
q^{\frac{n(\tau-1)}{\tau}}\prod_{i=1}^r (q^{\deg( f_i)} - n(b-1)-1)
\]
\end{proposition}
\begin{proof}
By Lemma~\ref{lem:trivial}, there are at most $q^{\deg(f_i)}$ choices for $g(\theta_i)$ for each $i = 1, \ldots, r$. By the final assertion of Theorem~\ref{thm:tfae-sociable}, an upper bound on the number of forbidden choices of $g(\theta_i)$ occurs when all of $c\theta_i^j$ are distinct for all $c \in \{1, \ldots, b-1\}$ and $j \in \{0,\ldots, n-1\}$. This gives $n(b-1)$ forbidden values for $g(\theta_i)$, and the further restriction $g(\theta_i)\ne 0$ together with Lemma \ref{lem:trivial} completes the proof.
\end{proof}

\begin{proposition}
Suppose $(n,q-1)=1$. Then the number of elements that are $(\alpha,b)$-sociable is
\[
q^{\frac{n(\tau-1)}{\tau}}\prod_{i=1}^r (q^{\deg(f_i)} - (b-1)\ord(\theta_i)-1),
\]
where $\theta_i$ is a root of $f_i$, $1 \leq i \leq r$. 
\end{proposition}

\begin{proof}
Since $(n,q-1)=1$, $x^n-1$ has only one root in $\fq$, namely $1$. Thus as each $\theta_i^j$ is an $n$-th root of $1$ (in some extension field), we have that $\theta_i^j\notin \fq$ for all $i$ and all $1<j<\ord(\theta_i)$. Therefore $\#\{c\theta_i^j:c \in \{1,\ldots,b-1\}, j\in \{0,\ldots,n-1\}\} = (b-1)\ord(\theta_i)$. As $g(\theta_i)\ne 0$, there are $q^{\deg(f_i)} - (b-1)\ord(\theta_i)-1$ choices for $g(\theta_i)$ for each $i$ for which $\phi_\alpha(g)$ is $(\alpha,b)$-sociable. The factor $q^{\frac{n(\tau-1)}{\tau}}$ follows from Lemma \ref{lem:trivial}.
\end{proof}

\begin{corollary}\label{cor:n=qs}
Suppose $n=q^s$. Then the number of elements that are $(\alpha,b)$-sociable is
\[
q^{q^s-q^{s-1}}(q - b).
\]
\end{corollary}

For a specific example of Corollary~\ref{cor:n=qs}, taking $q=n$, $b=2$, we get that there are $q^{q-1}(q-2)$ elements which are $(\alpha,2)$-sociable in $\mathbb{F}_{q^q}$.

\begin{corollary}
Suppose $n$ is prime, $n\notin \{p,q-1\}$, and let $x^n-1 = (x-1)f_2 \cdots f_r$. Then the number of elements that are $(\alpha,b)$-sociable is
\[
(q-b)\prod_{i=2}^r (q^{\deg(f_i)} - (b-1)n-1).
\]
\end{corollary}

In \cite{GoveGary}, focus is applied to the case $b=2$, the case of $(\alpha,2)$-lonely/sociable elements. We now apply Theorem~\ref{thm:tfae-sociable} to this situation.

\begin{proposition}\label{prop:linearsplit}
Suppose $n|(q-1)$. Then the number of elements that are $(\alpha,2)$-sociable is
\begin{equation}
\prod_{i=1}^n \left(q- \frac{n}{(i,n)}-1\right).\label{eqn:coprime}
\end{equation}
\end{proposition}

\begin{proof}
As $n|(q-1)$, $x^n-1$ factorises in to a product of distinct linear factors over $\fq$. Let $f_i = x-\theta_i$. Then $\beta$ is $(\alpha,2)$-sociable if and only if $g(\theta_i)\ne 0, \theta_i^j$ for any $j$. Thus the number of forbidden choices for $g(\theta_i)$ is $\ord(\theta_i)+1$. Letting $\theta$ be a primitive $n$-th root of unity in $\fq$, and letting $\theta_i=\theta^i$, then $\ord(\theta_i)=\frac{n}{(i,n)}$ and the result follows.
\end{proof}

\begin{remark}
Note that Formula~\eqref{eqn:coprime} is not true in general. Issues arise when there exist $c_1,c_2\in \{0,\ldots,b-1\}$ such that $c_1=c_2\theta_i^j$, in which case $\#\{c\theta_i^j:c \in \{0,\ldots,b-1\}, j\in \{0,\ldots,n-1\}\}$ is more difficult to calculate. The conditions of the previous two theorems were chosen to avoid this possibility.
\end{remark}

The following example of Proposition~\ref{prop:linearsplit} provides an answer to the first open question left in~\cite{GoveGary}. 

\begin{example}
Suppose $n=3$, and suppose $x^3-1$ factors into distinct linear factors over $\fq$, say $x^3-1=(x-1)(x-\lambda)(x-\mu)$; equivalently, if $q \equiv 1\pmod{3}$. Then $\phi_\alpha(g)$ has normal $\alpha$-depth $2$ if and only if $\{0,1\}\cap\{g(1),g(\lambda),g(\mu)\}=\emptyset$. Similarly, $\phi_\alpha(g)^{q^i}$ has normal $\alpha$-depth $2$ if and only if $\{0,1\}\cap\{g(1),\lambda^i g(\lambda),\mu^i g(\mu)\}=\emptyset$. Since a polynomial of degree at most three is uniquely determined by its evaluation at three different elements of $\fq$, then there are $(q-2)^3$ elements of $\alpha$-depth $2$, of which $(q-2)(q-4)^2$ are not lonely. Thus there are $4(q-2)(q-3)$ lonely elements.
\end{example}

We can also apply Proposition~\ref{prop:linearsplit} to provide a partial answer to the second open question in~\cite{GoveGary}. 

\begin{example}\label{ex:p=n+1}
Suppose $n = 4$, $q=5$, then $x^4-1 = (x-1)(x-2)(x-3)(x-4)$, with $\ord(1) = 1$, $\ord(4) = 2$ and $\ord(2)=\ord(3)=4$. A direct application of Proposition~\ref{prop:linearsplit} shows that there are no $(\alpha,2)$-sociable elements. 
\end{example}

Example~\ref{ex:p=n+1} generalizes in an obvious way.

\begin{proposition}
Let $q = n+1$, then there are no $(\alpha, 2)$-sociable elements of $\fqn$. 
\end{proposition}

\begin{proof}
Let $\theta$ be a primitive element in $\fq$. Then $x^n-1=x^{q-1}-1=\prod_{\lambda\in \fq^*}(x-\lambda)=\prod_{i=0}^{q-2}(x-\theta^i) $. Hence for $\beta=\phi_\alpha(g)$ to be $(\alpha, 2)$-sociable it would require that $g(\theta)\ne 0$ and $g(\theta)\ne \theta^i$ for any $0\leq i\leq q-2$, which is impossible as $g(\theta)\in \fq$.
\end{proof}

The following was proved in \cite[Proposition 4.3]{GoveGary}. We include an alternative proof here.

\begin{proposition}
Suppose $\frac{x^n-1}{x-1}$ is irreducible over $\fq$. Then the number of elements that are $(\alpha,2)$-sociable is
\[
(q-2)(q^{n-1}-n-1),
\]
and the number of elements that are $(\alpha,2)$-lonely is
\[
(q-2)(n-1).
\]
\end{proposition}

\begin{proof}
Recall that $\frac{x^n-1}{x-1}$ is irreducible over $\fq$ if and only if $q$ is primitive modulo $n$. Then $\{\theta^{q^i}\colon i = 0, \ldots, n-2\} = \{\theta^i\colon i = 1, \ldots, n-1\}$ is the set of distinct roots of $\frac{x^n-1}{x-1}$. Thus, an element $\phi_\alpha(g)$ is $(\alpha,2)$-sociable if and only if $g(1) \neq 0,1$ and $g(\theta) \neq \theta^i$ for $i = 1,2, \ldots, n-1$. Hence, there are $(q-2)(q^{n-1}-n+1)$ elements that are $(\alpha,2)$-sociable and $(q-2)(n-1)$ lonely elements in $\fqn$.  
\end{proof}

The following example examines two cases of $(\alpha, 3)$-sociable elements, giving the first directions towards the third open problem in~\cite{GoveGary}.

\begin{example}
\begin{enumerate}
\item Consider the case $q=7$, $n=3$, $b=3$. Then $x^n-1=(x-1)(x-2)(x-4)$, and $2^3=1$. Now the set $\{c \theta^j:c\in \{0,1,2\},j\in \{0,1,2\}\}$ is equal to $\{0,1,2\}$ for $\theta=1$ and $\{0,1,2,4\}$ for $\theta=2,4$. Thus the number of $(\alpha,3)$-sociable elements is $(7-3)(7-4)^2=36$. 

\item In the case $q=13$, $n=3$, $b=3$, we have $x^n-1=(x-1)(x-3)(x-9)$. Now the set $\{c \theta^j:c\in \{0,1,2\},j\in \{0,1,2\}\}$ is equal to $\{0,1,2\}$ for $\theta=1$ and $\{0, 1, 2, 3, 5, 6, 9\}$ for $\theta=3,9$. Thus the number of $(\alpha,3)$-sociable elements is $(13-3)(13-7)^2=36$. 
\end{enumerate}
\noindent These two examples illustrates how extra care must be taken when an element of $\{0,\ldots,b-1\}$ is a nontrivial $n$-th root of unity.
\end{example}

\section{Conclusions and future directions}
In this paper, we study a generalization of normal elements of depth $b$, as presented in~\cite{GoveGary}. Since depth is not invariant under conjugation, we further analyze the depth of the conjugates of normal elements. 

The notion of depth readily lends itself to further generalization. One such ``natural'' generalization is as follows. Given some total ordering $\mathcal{O}$ of the elements of $\fqn$, say $\mathcal{O} = \{o_0, o_1, \ldots, o_{q^n-1}\}$, an element $\beta \in \fqn$ has normal $(\mathcal{O},\alpha)$-depth $b$ if $\beta - o_0\alpha, \beta - o_1\alpha, \ldots, \beta - o_{b-1}\alpha$ are simultaneously normal. Here, $\beta$ has $\alpha$-depth $b$ if $o_i = i$ for $i = 0, \ldots, b-1$. Some interesting questions here occur when $\alpha$ is a normal element of $\fqn$ over $\fq$ and $\mathcal{O}_\zeta = (0, \zeta, \zeta^2, \ldots, \zeta^{q^n-2})$, for a primitive element $\zeta \in \fqn$. Determining conditions for which $\beta$ has $(\mathcal{O}_\zeta, \alpha)$-depth $2$, or statistics on the possible values of $b$ for which $\beta$ has $(\mathcal{O}_\zeta, \alpha)$-depth $b$ is the subject of future work.


\begin{thebibliography}{99}

\bibitem{GoveGary}
G. Effinger and G. L. Mullen, Two extended Euler functions with applications to latin squares and bases of finite field extensions, \emph{Bulletin of the Institute for Combinatorics and its Applications}, {\bf 85} (2019), 92-111. 

\bibitem{LS}
H. W. Lenstra and R. Schoof, Primitive Normal Bases for Finite Fields, \emph{Mathematics of Computation}, {\bf 48} (1987), 217-231.

\bibitem{prim-1-normal}
L. Reis and D. Thomson, Existence of primitive $1$-normal elements in finite fields, \emph{Finite Fields and Their Applications}, {\bf 51} (2018), 238-269. 

\bibitem{Steel}
A. Steel, A New Algorithm for the Computation of Canonical Forms of Matrices over Fields, \emph{Journal of Symbolic Computation}, {\bf 24} (1997), 409-432.

\end{thebibliography}
\end{document}